\documentclass[12pt, a4paper, leqno]{amsart}
\usepackage{esint} 
\usepackage{graphicx}
\usepackage{url}

\usepackage{amsthm}
\usepackage{amsmath}
\usepackage{amsfonts}
\usepackage{amssymb}
\usepackage{amscd}
\usepackage{mathrsfs}
\usepackage{enumerate}

\usepackage{verbatim}

\DeclareFontFamily{U}{mathx}{\hyphenchar\font45}
\DeclareFontShape{U}{mathx}{m}{n}{
      <5> <6> <7> <8> <9> <10>
      <10.95> <12> <14.4> <17.28> <20.74> <24.88>
      mathx10
      }{}
\DeclareSymbolFont{mathx}{U}{mathx}{m}{n}
\DeclareFontSubstitution{U}{mathx}{m}{n}
\DeclareMathAccent{\widecheck}{0}{mathx}{"71}
\DeclareMathAccent{\wideparen}{0}{mathx}{"75}

\usepackage[utf8]{inputenc}
\usepackage[english]{babel}
\usepackage{anysize}
\usepackage{tikz-cd}

\usepackage{lipsum}

\newcommand{\N}{\mathbb{N}}
\newcommand{\R}{\mathbb{R}}

\newcommand{\Z}{\mathbb{Z}}

\newcommand{\rk}{\operatorname{rk}}

\newtheorem{teor}{Theorem}

\newtheorem{lemma}{Lemma}

\theoremstyle{remark}

\theoremstyle{remark}
\newtheorem{remark}{Remark}

\usepackage{dsfont}
\usepackage{centernot}

\newcommand{\nwc}{\newcommand}
\nwc{\nwt}{\newtheorem}
\nwt{theo}{Theorem}
\nwt{coro}[theo]{Corollary}
\nwt{ex}[theo]{Example}
\nwt{prop}[theo]{Proposition}
\nwt{defin}[theo]{Definition}
\nwt{rem}[theo]{Remark}

\nwc{\mf}{\mathbf} 
\nwc{\blds}{\boldsymbol} 
\nwc{\ml}{\mathcal} 

\nwc{\lam}{\lambda}
\nwc{\del}{\delta}
\nwc{\Del}{\Delta}
\nwc{\Lam}{\Lambda}
\nwc{\elll}{\ell}

\nwc{\IA}{\mathbb{A}} 
\nwc{\IB}{\mathbb{B}} 
\nwc{\IC}{\mathbb{C}} 
\nwc{\ID}{\mathbb{D}} 
\nwc{\IE}{\mathbb{E}} 
\nwc{\IF}{\mathbb{F}} 
\nwc{\IG}{\mathbb{G}} 
\nwc{\IH}{\mathbb{H}} 
\nwc{\IN}{\mathbb{N}} 
\nwc{\IP}{\mathbb{P}} 
\nwc{\IQ}{\mathbb{Q}} 
\nwc{\IR}{\mathbb{R}} 
\nwc{\IS}{\mathbb{S}} 
\nwc{\IT}{\mathbb{T}} 
\nwc{\IZ}{\mathbb{Z}} 

\def\bbleft{{\mathchoice {[\mskip-3mu {[}} {[\mskip-3mu {[}}{[\mskip-4mu {[}}{[\mskip-5mu {[}}}}
\def\bbright{{\mathchoice {]\mskip-3mu {]}} {]\mskip-3mu {]}}{]\mskip-4mu {]}}{]\mskip-5mu {]}}}}
\nwc{\setK}{\bbleft 1,K \bbright}
\nwc{\setN}{\bbleft 1,\cN \bbright}
 
\nwc{\va}{{\bf a}}
\nwc{\vb}{{\bf b}}
\nwc{\vc}{{\bf c}}
\nwc{\vd}{{\bf d}}
\nwc{\ve}{{\bf e}}
\nwc{\vf}{{\bf f}}
\nwc{\vg}{{\bf g}}
\nwc{\vh}{{\bf h}}
\nwc{\vi}{{\bf i}}
\nwc{\vI}{{\bf I}}
\nwc{\vj}{{\bf j}}
\nwc{\vk}{{\bf k}}
\nwc{\vl}{{\bf l}}
\nwc{\vm}{{\bf m}}
\nwc{\vM}{{\bf M}}
\nwc{\vn}{{\bf n}}
\nwc{\vo}{{\it o}}
\nwc{\vp}{{\bf p}}
\nwc{\vq}{{\bf q}}
\nwc{\vr}{{\bf r}}
\nwc{\vs}{{\bf s}}
\nwc{\vt}{{\bf t}}
\nwc{\vu}{{\bf u}}
\nwc{\vv}{{\bf v}}
\nwc{\vw}{{\bf w}}
\nwc{\vx}{{\bf x}}
\nwc{\vy}{{\bf y}}
\nwc{\vz}{{\bf z}}
\nwc{\bal}{\blds{\alpha}}
\nwc{\bep}{\blds{\epsilon}}
\nwc{\barbep}{\overline{\blds{\epsilon}}}
\nwc{\bnu}{\blds{\nu}}
\nwc{\bmu}{\blds{\mu}}
\nwc{\bet}{\blds{\eta}}

\nwc{\bk}{\blds{k}}
\nwc{\bm}{\blds{m}}
\nwc{\bM}{\blds{M}}
\nwc{\bp}{\blds{p}}
\nwc{\bq}{\blds{q}}
\nwc{\bn}{\blds{n}}
\nwc{\bv}{\blds{v}}
\nwc{\bw}{\blds{w}}
\nwc{\bx}{\blds{x}}
\nwc{\bxi}{\blds{\xi}}
\nwc{\by}{\blds{y}}
\nwc{\bz}{\blds{z}}

\nwc{\cA}{\ml{A}}
\nwc{\cB}{\ml{B}}
\nwc{\cC}{\ml{C}}
\nwc{\cD}{\ml{D}}
\nwc{\cE}{\ml{E}}
\nwc{\cF}{\ml{F}}
\nwc{\cG}{\ml{G}}
\nwc{\cH}{\ml{H}}
\nwc{\cI}{\ml{I}}
\nwc{\cJ}{\ml{J}}
\nwc{\cK}{\ml{K}}
\nwc{\cL}{\ml{L}}
\nwc{\cM}{\ml{M}}
\nwc{\cN}{\ml{N}}
\nwc{\cO}{\ml{O}}
\nwc{\cP}{\ml{P}}
\nwc{\cQ}{\ml{Q}}
\nwc{\cR}{\ml{R}}
\nwc{\cS}{\ml{S}}
\nwc{\cT}{\ml{T}}
\nwc{\cU}{\ml{U}}
\nwc{\cV}{\ml{V}}
\nwc{\cW}{\ml{W}}
\nwc{\cX}{\ml{X}}
\nwc{\cY}{\ml{Y}}
\nwc{\cZ}{\ml{Z}}

\nwc{\fA}{\mathfrak{a}}
\nwc{\fB}{\mathfrak{b}}
\nwc{\fC}{\mathfrak{c}}
\nwc{\fD}{\mathfrak{d}}
\nwc{\fE}{\mathfrak{e}}
\nwc{\fF}{\mathfrak{f}}
\nwc{\fG}{\mathfrak{g}}
\nwc{\fH}{\mathfrak{h}}
\nwc{\fI}{\mathfrak{i}}
\nwc{\fJ}{\mathfrak{j}}
\nwc{\fK}{\mathfrak{k}}
\nwc{\fL}{\mathfrak{l}}
\nwc{\fM}{\mathfrak{m}}
\nwc{\fN}{\mathfrak{n}}
\nwc{\fO}{\mathfrak{o}}
\nwc{\fP}{\mathfrak{p}}
\nwc{\fQ}{\mathfrak{q}}
\nwc{\fR}{\mathfrak{r}}
\nwc{\fS}{\mathfrak{s}}
\nwc{\fT}{\mathfrak{t}}
\nwc{\fU}{\mathfrak{u}}
\nwc{\fV}{\mathfrak{v}}
\nwc{\fW}{\mathfrak{w}}
\nwc{\fX}{\mathfrak{x}}
\nwc{\fY}{\mathfrak{y}}
\nwc{\fZ}{\mathfrak{z}}

\nwc{\tA}{\widetilde{A}}
\nwc{\tB}{\widetilde{B}}
\nwc{\tE}{E^{\vareps}}
\nwc{\tk}{\tilde k}
\nwc{\tN}{\tilde N}
\nwc{\tP}{\widetilde{P}}
\nwc{\tQ}{\widetilde{Q}}
\nwc{\tR}{\widetilde{R}}
\nwc{\tV}{\widetilde{V}}
\nwc{\tW}{\widetilde{W}}
\nwc{\ty}{\tilde y}
\nwc{\teta}{\tilde \eta}
\nwc{\tdelta}{\tilde \delta}
\nwc{\tlambda}{\tilde \lambda}
\nwc{\ttheta}{\tilde \theta}
\nwc{\tvartheta}{\tilde \vartheta}
\nwc{\tPhi}{\widetilde \Phi}
\nwc{\tpsi}{\tilde \psi}
\nwc{\tmu}{\tilde \mu}

\nwc{\To}{\longrightarrow} 

\nwc{\ad}{\rm ad}
\nwc{\eps}{\varepsilon}
\nwc{\ep}{\epsilon}
\nwc{\vareps}{\varepsilon}

\def\ep{\epsilon}

\def\sq2{\sqrt{2}}

\def\t2{{\mathbb T}^2}
\def\s2{{\mathbb S}^2}

\def\N{\mathbb{N}}

\def\R{\mathbb{R}}

\def\Z{\mathbb{Z}}

\nwc{\lap}{\bigtriangleup}
\nwc{\rest}{\restriction}
\nwc{\Diff}{\operatorname{Diff}}
\nwc{\diam}{\operatorname{diam}}
\nwc{\Res}{\operatorname{Res}}
\nwc{\Spec}{\operatorname{Sp}}
\nwc{\Vol}{\operatorname{Vol}}
\nwc{\Op}{\operatorname{Op}}
\nwc{\supp}{\operatorname{supp}}
\nwc{\Span}{\operatorname{span}}
\nwc{\weaksc}{\overset{\ast}{\rightharpoonup}}
\nwc{\charf}{\mathds{1}}
\nwc{\notimplies}{\centernot\implies}

\nwc{\dia}{\varepsilon}
\nwc{\cut}{f}
\nwc{\qm}{u_\hbar}

\def\hto0{\xrightarrow{\hbar\to 0}}

\def\rto0{\xrightarrow{r\to 0}}

\nwc{\la}{\langle}
\nwc{\ra}{\rangle}
\nwc{\lp}{\left(}
\nwc{\rp}{\right)}

\nwc{\bequ}{\begin{equation}}
\nwc{\be}{\begin{equation}}
\nwc{\ben}{\begin{equation*}}
\nwc{\bea}{\begin{eqnarray}}
\nwc{\bean}{\begin{eqnarray*}}
\nwc{\bit}{\begin{itemize}}
\nwc{\bver}{\begin{verbatim}}

\nwc{\eequ}{\end{equation}}
\nwc{\ee}{\end{equation}}
\nwc{\een}{\end{equation*}}
\nwc{\eea}{\end{eqnarray}}
\nwc{\eean}{\end{eqnarray*}}
\nwc{\eit}{\end{itemize}}
\nwc{\ever}{\end{verbatim}}

\title{Localization and delocalization of eigenmodes of Harmonic Oscillators}

\author[Víctor Arnaiz]{Víctor Arnaiz}

\address{Université Paris-Saclay, CNRS, Laboratoire de Mathématiques d'Orsay, 91405, Orsay, FRANCE.}

\email{victor.arnaiz@universite-paris-saclay.fr}

\author[Fabricio Macià]{Fabricio Macià}

\address{M$^2$ASAI, Universidad Politécnica de Madrid.  \newline
ETSI Navales, Avda. de la Memoria, 4. 28040 Madrid, SPAIN.}

\email{fabricio.macia@upm.es}

\begin{document}

\begin{abstract}
We characterize quantum limits and semi-classical measures corresponding to sequences of eigenfunctions for systems of coupled quantum harmonic oscillators with arbitrary frequencies. The structure of the set of semi-classical measures turns out to depend strongly on the arithmetic relations between frequencies of each decoupled oscillator. In particular, we show that as soon as these frequencies are not rational multiples of a fixed fundamental frequency, the set of semi-classical measures is not convex and therefore, infinitely many measures that are invariant under the classical harmonic oscillator are not semi-classical measures.
\end{abstract}

\maketitle

\section{Introduction}

Understanding the distribution of high-frequency eigenfunctions of elliptic operators, in particular the presence of scarring (concentration on subsets of zero Lebesgue measure), has been the subject of intensive study in the past fifty years. 
In this note, we discuss this problem for eigenfunctions of Schrödinger operators acting on $ L^2(\R^{d})$ of the form:
\begin{equation}\label{e:defp}
\widehat{P}:=- \frac{1}{2}  \Delta_x + \frac{1}{2} Q.
\end{equation}
with $Q$ a positive definite quadratic form on $\R^d$.

Since $\widehat{P}$ has compact resolvent in $L^2(\R^d)$, its spectrum is discrete and unbounded. Moreover, the non-compactness of $\R^d$ implies that the $ L^2(\R^{d})$-mass of any sequence of normalized eigenfunctions with eigenvalues tending to infinity will escape any bounded set, both in position and momentum variables. In order to observe eventual concentration-like behavior on the eigenfunctions, it is natural to rescale the problem. Consider the unitary operators:
\[
T_\hbar: L^2(\R^{d})\ni\psi \longmapsto \frac{1}{\hbar^{d/4}}\psi\left(\frac{\cdot}{\sqrt{\hbar}}\right)\in L^2(\R^{d}), \quad \hbar>0;
\]
then  
\[
\widehat{P}_\hbar :=- \frac{1}{2} \hbar^2 \Delta_x + \frac{1}{2} Q=\hbar  T_\hbar \widehat{P} T_\hbar^*.
\] 
A further unitary conjugation, which essentially amounts to diagonalizing the quadratic form $Q$, transforms $\widehat{P}_\hbar$ into a semi-classical quantum harmonic oscillator:
\begin{equation}
\label{quantum_harmonic_oscillator}
\widehat{H}_\hbar: = \frac{1}{2} \sum_{j=1}^d \omega_j \big( - \hbar^2 \partial_{x_j}^2 +  x_j^2 \big),
\end{equation}
where $\omega_1^2\leq \ldots \leq \omega_d^2 $ are the eigenvalues of $Q$ and the frequencies are chosen to satisfy $\omega_j >0$ for $j=1,\hdots,d$. In particular, the spectrum of this operator satisfies:
\[
\Spec(\widehat{H}_\hbar)=\hbar \Spec(\widehat{P}).
\]

The vector 
\[
\omega := (\omega_1, \ldots , \omega_d) \in \R^d_+
\]
is called the \textit{vector of frequencies}; its arithmetic properties will play an important role in the sequel.

The above considerations lead to naturally consider the asymptotic properties of eigenfunctions in the \textit{semi-classical limit}, that is, letting the eigenvalue $\lambda_n$ tend to infinity and the semi-classical parameter $\hbar_n$ tend to zero while keeping the energy $\hbar_n\lambda_n$ constant. More explicitly, we will consider sequences of eigenfunctions $(\Psi_\hbar)$ satisfying 
\begin{equation}
\label{e:stationary_problem}
\widehat{H}_\hbar \, \Psi_\hbar = \lambda_\hbar \, \Psi_\hbar, \quad \Vert \Psi_\hbar \Vert_{L^2(\R^d)} = 1,\quad \text{where }\lambda_\hbar \to 1 \text{ as } \hbar \to 0^+.
\end{equation}
In particular, given a sequence $(\Psi_\hbar)$ of solutions to \eqref{e:stationary_problem}, we are interested in investigating the structure of the accumulation points of the sequence of probability densities $(\vert \Psi_\hbar \vert^2 dx)$ (in the weak-$\star$ topology of Radon measures on $\R^d$). These accumulation points are again probability measures in $\R^d$ and are sometimes called \textit{quantum limits}. 

The existence of quantum limits that are singular with respect to the Lebesgue measure is known as \textit{eigenfunction scarring}; understanding which systems exhibit this type of behavior is a notoriously difficult problem. 

Quantum limits can be defined for general elliptic operators with compact resolvent, for instance a Schrödinger operator  on a compact manifold or a bounded domain of Euclidean space. In this setting, they have been completely characterized in relatively few cases: the Laplacian on spheres \cite{Jak_Zel99}, and more generally on compact rank-one symmetric spaces \cite{MaciaZoll} or space forms \cite{AM:10}; and the Laplacian on the two dimensional torus $\mathbb{T}^2=\R^2/\mathbb{Z}^2$ \cite{Jak97}. Spheres and their relatives exhibit the strongest form of scarring: any uniform measure supported on a geodesic is a quantum limit. The situation is completely different for tori: quantum limits are always absolutely continuous (although concentration in the momentum variable is possible, this will be discussed below). 

Other settings that are relatively well-understood include the Laplacian on Zoll manifolds \cite{MaciaRiviere16} (manifolds whose geodesics are closed), where scarring takes place, but usually in a weaker form: there are uncountable families of geodesics onto which no eigenfunctions scar. The Laplacian on manifolds with non-degenerate completely integrable geodesic flow has also been studied \cite{An_Fer_Mac15, TothIntegrable, Wunsch2012}; scarring cannot take place on regions that can be parametrized by action-angle coordinates. A similar phenomenology can be found for the Dirchlet Laplacian on planar domains with completely integrable billiard flow: rational polygonal domains \cite{HHM, MarRud12} or the Euclidean disk \cite{ALM:16}. 

Not much is known in the case the Laplacian on a manifold whose geodesic flow is a small perturbation of a completely integrable system (KAM systems); however, there have been some interesting recent developments in that direction \cite{Ar18,Gomes18,Gom_Hass18}, where some form scarring has been proved. On the other side of the dynamics lie manifolds of negative curvature; characterizing the set of Quantum Limits in this case is part of the Quantum Unique Ergodicity conjecture \cite{RudSar}: the conjecture implies that the only quantum limit is the Riemannian volume. The literature is vast in this setting; see, among many others, \cite{NAQUE,AnNon,BouLin,CdV:85,DyatlovJin,Hass10,LindenQUE,Sni, Zel87}.

As all the preceding examples indicate, a central role in this problem is played by a Hamiltonian dynamical system (the geodesic flow) acting on phase-space (the cotangent bundle of a manifold). Here, the relevant object is the classical harmonic oscillator:
\begin{equation}
\label{classical_harmonic_oscillator}
H(x,\xi) = \frac{1}{2}  \sum_{j=1}^d  \omega_j \big( \xi_j^2 +  x_j^2 \big), \quad (x,\xi) \in T^*\R^d=\R^{d}\times\R^d,
\end{equation}
whose induced Hamiltonian flow will be denoted by $\phi^H_t$.  
It is therefore natural to lift to phase-space $T^*\R^d = \R^d_x \times \R^d_\xi$ all objects of interest, since this is where the classical Hamiltonian flow is defined. Here we will lift the density $|\psi|^2$ of a function $\psi \in L^2(\R^d)$ through its associated \textit{Wigner function} $W^\hbar_{\psi}\in L^2(\R^d\times\R^d)$, which is defined by:
\[
W^\hbar_{\psi}(x,\xi):=\int_{\R^d}\psi\left(x-\frac{\hbar v}{2}\right)\overline{\psi\left(x+\frac{\hbar v}{2}\right)}e^{i\xi\cdot v}\frac{dv}{(2\pi)^d}.
\]
That this is actually a lift follows from this useful property: for every $a \in \mathcal{C}^\infty(\R^{d}\times\R^d)$ that is bounded, 
$$
\int_{\R^{d}\times\R^d}a(x,\xi)W^\hbar_{\psi}(x,\xi)\,dx\,d\xi= \big \langle \psi , \Op_\hbar(a) \psi \big \rangle_{L^2(\R^d)}, 
$$
where $\Op_\hbar(a)$ denotes the semiclassical Weyl quantization of the symbol $a$ (see \cite{Dim_Sjo99, FollandPhaseSpace, Zwobook}, for instance) and $\langle \cdot, \cdot \rangle_{L^2(\R^d)}$ is the scalar product on $L^2(\R^d)$. When $a$ does not depend on the $\xi$ variable, the operator $\Op_\hbar(a)$ is the multiplication operator by $a$, and therefore:
\begin{equation}\label{e:lift}
\int_{\R^d}W^\hbar_{\psi}(\cdot,\xi)d\xi =|\psi|^2. 
\end{equation}

If $(\Psi_\hbar)$ solves \eqref{e:stationary_problem} then $(W^\hbar_{\Psi_\hbar})$ is bounded in the space of tempered distributions $\cS'(\R^d\times\R^d)$. Moreover, every accumulation point of $(W_{\Psi_\hbar}^\hbar)$ belongs to the set $\mathcal{M}(H)$ of Radon probability measures on $\R^d\times\R^d$ that are concentrated on the level set $H^{-1}(1)$ and that are invariant by the Hamiltonian flow $\phi_t^H$. See for instance \cite{GerardMesuresSemi91} for proofs of these facts. 

We will denote by $\mathcal{M}(\widehat{H}_\hbar)$ the set of all accumulation points obtained in this way, as the sequence $(\Psi_\hbar)$ varies among those satisfying \eqref{e:stationary_problem}. Its elements are called \textit{semi-classical measures}. 

The lift property \eqref{e:lift} still holds after taking limits. Therefore, the set of quantum limits is obtained from $\mathcal{M}(\widehat{H}_\hbar)$ by projecting semi-classical measures onto the $x$-variable.

In order to state our result, we recall some basic facts on the dynamics of the harmonic oscillator. Set
\[ 
H_j(x,\xi) = \frac{1}{2} \big( \xi_j^2 + x_j^2 \big), \quad j \in \{1, \ldots , d \}.
\]
We can write $H$ as a function of $H_1, \ldots, H_d$ by putting
\begin{equation}
\label{e:linear_for_H}
H = \mathcal{L}_\omega(H_1, \ldots, H_d),
\end{equation}
where $\mathcal{L}_\omega : \R^d_+ \to \R$ is the linear form defined by $\mathcal{L}_\omega(E) = \omega \cdot E$. 
Since $\{ H_j , H_k \} = 0$ for every $j,k\in \{1, \ldots , d \}$, the Hamiltonian flow of $H$ can be written as 
\[
\phi_t^H(z)= \phi_{\omega_d t}^{H_d} \circ \cdots \circ \phi_{\omega_1 t}^{H_1}(z), \;z = (x,\xi)\in\R^{2d},
\] 
where $\phi_{t}^{H_j}$ denotes the flow of $H_j$. These flows are totally explicit, they act as a rotation of angle $t$ on the plane $(x_j,\xi_j)$. If one identifies points $(x_j,\xi_j)$ in this plane to the complex numbers $z_j:=x_j+i\xi_j$, then $\phi_{t}^{H_j}$ acts on this plane as $e^{-it}z_j$ and fixes the points in its orthogonal complement. 

These flows are $2\pi$-periodic, therefore Konecker's theorem shows that the orbit of $\phi^H_t$ of any point $z_0\in \cX$, where
\[
\cX:=\{z\in H^{-1}(1)\,:\, H_j(z)>0, \;j=1,\hdots, d\},
\]
is dense in a torus of dimension 
\[
d_\omega:= \dim S_\omega,
\]
where
\[
S_\omega=\langle \omega_1,\hdots,\omega_d \rangle_\IQ
\]
is the linear subspace of $\R$, viewed as a vector space over the rationals, spanned by the frequencies. 

When $d_\omega = d$, orbits corresponding to points $z_0\in \cX$ are dense in a $d$-dimensional Lagrangian torus. When $d_\omega=1$, the flow $\phi^H_t$ is periodic. More precisely, the following holds.

\begin{remark}\label{r:period}
Suppose $v\in\IR^d$ is an arbitrary vector of frequencies satisfying $\dim S_v=1$, then the flow of the harmonic oscillator with Hamiltonian $\sum_{j=1}^d v_j H_j$ is periodic of period 
\[
T_v:=2\pi k_v/[v],
\]
where $[v]:=\min{\{|v_j|\,:\,v_j\neq 0,\, j=1,\hdots,d\}}$ and $k_v$ is the least positive integer such that $k_v [v]^{-1}v\in\Z^d$. 
\end{remark}

The Hamiltonian of the harmonic oscillator admits a decomposition that is particularly useful to our purposes. Let $\{v_1, \ldots, v_{d_\omega} \}$ be a basis of $S_\omega$. Then, for every $1 \leq n \leq d_\omega$, there exists $\nu_n \in \mathbb{Q}^{d}$ such that
$$
\omega = \sum_{n=1}^{d_\omega} v_n \nu_{n}.
$$
Note that the vectors $\nu_n$, $n=1,\hdots,d_\omega$ must be linearly independent and generate a linear subspace that only depends on $\omega$. Write for every $n=1,\hdots,d_\omega$,
\begin{equation}\label{e:decper}
\cH_n:= v_n \sum_{j=1}^d \nu_{n,j} H_j,
\end{equation}
so that $H=\sum_{n=1}^{d_\omega}\cH_n$.  The flows $\phi^{\cH_n}_t$ of $\cH_n$ with $n=1,\hdots,d_\omega$ are periodic since the  vector of frequencies of $\cH_n$ is $v_n\nu_n$ which satisfies $\dim S_{v_n\nu_n}=1$. Moreover, those measures invariant by the flow of the harmonic oscillator are precisely those simultaneously invariant by all $\phi^{\cH_n}_t$:
\begin{equation}\label{e:siminv}
\mu\in\cM(H)\quad \iff \quad (\phi^{\cH_n}_t)_*\mu=\mu,\quad\forall t\in \R,\;n=1,\hdots,d_\omega.
\end{equation}

Our main result characterizes the set of semi-classical measures associated to eigenstates of the Hamiltonian $\widehat{H}_\hbar$. It turns out that it consists precisely of those invariant measures supported on the intersection of the level sets of the Hamiltonians $\cH_n$, $n=1,\hdots,d_\omega$. Let $\mathcal{H} := (\mathcal{H}_1, \ldots, \mathcal{H}_{d_\omega})$ and write:
\[
\Sigma_\cH:=\{\cH(z)\,:\,z\in H^{-1}(1)\}.
\]
Note that $\Sigma_\cH$ is a compact subset of $\IR^{d_\omega}$ and $\cH^{-1}(E)\subseteq H^{-1}(1)$ for every $E\in\Sigma_\cH$.
\begin{teor}
\label{non-perturbed-eigenfunctions}
A probabily measure $\mu$ on $\R^d\times\R^d$ satisfies $\mu\in \mathcal{M}(\widehat{H}_\hbar) $ if and only if there exists $E\in \Sigma_\cH$ such that 
\[
\supp\mu\subseteq \cH^{-1}(E),
\]
and
\[
(\phi^{\cH_n}_t)_*\mu=\mu,\quad\forall t\in \R,\;n=1,\hdots,d_\omega.
\]
\end{teor}

Note that, since the commuting Hamiltonian vector fields $X_{\cH_n}$ of $\cH_n$ are linearly independent when restricted to $\cX$, Theorem \ref{non-perturbed-eigenfunctions} implies that the restriction of the measures in $\mathcal{M}(\widehat{H}_\hbar)$ to  $\cX$ are smoother as $d_\omega$ increases. 
\begin{remark}
Notice that $\mathcal{H}$ depends on a particular choice of the basis $\{ v_1, \ldots, v_{d_\omega} \}$. However, 
the resulting partition of $H^{-1}(1)$ in level sets $\mathcal{H}^{-1}(E)$ for $E\in \Sigma_\cH$ does not depend on this choice.
\end{remark}
\begin{remark}
This partition of $H^{-1}(1)$ is non-trivial as soon as $d_\omega>1$. Therefore, Theorem \ref{non-perturbed-eigenfunctions} implies that $\mathcal{M}(\widehat{H}_\hbar)$ is not a convex set, unless $d_\omega=1$. In other words:
$$
\mathcal{M}(\widehat{H}_\hbar)=\mathcal{M}(H) \quad \iff \quad \phi^H_t \text{ is periodic }.
$$
\end{remark}

The proof of Theorem \ref{non-perturbed-eigenfunctions} follows from an argument on the propagation of wave-packets which is inspired in \cite{DBievre93} (see Lemma \ref{concentration_minimal_set} in Section \ref{weak_limits_eigenfunctions}).   In Section \ref{proof} we show how to conclude the proof of the theorem using the decomposition \eqref{e:decper}.

The particular case of an isotropic harmonic oscillator, \textit{i.e.} with $\omega=(1,\hdots,1)$, was first analysed in \cite{Arnaiz_tesis}. See also \cite{Studnia19} for a different proof which is based on the invariance by the metaplectic representation of the unitary group in $\IC^{2d}$ (this strategy of proof can only be implemented in the isotropic case)  and the results in \cite{Ojeda_Villegas10}, that show scarring of eigenfunctions on periodic orbits. It is interesting to note that our approach here is necessarily very different, since the proofs we just mentioned cannot be carried out even in the case of a periodic anisotropic harmonic oscillator. 
 
\subsection*{Acknowledgments}
The authors would like to thank Stéphane Nonnenmacher, Gabriel Rivière, and Carlos Villegas-Blas for fruitful discussions on semiclassical harmonic oscillators.  We are also grateful to Yves Colin de Verdière for pointing out a gap in a previous version of the article. Victor Arnaiz has been supported by a predoctoral grant from Fundación La Caixa - Severo Ochoa International Ph.D. Program at the Ins\-tituto de Ciencias Matemáticas (ICMAT-CSIC-UAM-UC3M-UCM), and is currently supported by the European Research Council (ERC) under the European Union's Horizon 2020 research and innovation programme (grant agreement No. 725967). Both authors have been partially supported by grant MTM2017-85934-C3-3-P (MINECO, Spain).

\section{Scarring on minimal invariant tori}
\label{weak_limits_eigenfunctions}

We first show how to construct sequences of joint eigenfunctions of the periodic Hamiltonians $\cH_n$ that concentrate on minimal invariant tori. Let $T_n:=T_{v_n\nu_n}$ denote the period of $\phi_t^{\cH_n}$ as defined in Remark \ref{r:period}. Write $\mathcal{R}_\omega := [0,T_1) \times \cdots \times [0,T_{d_\omega})$, and let $\vert \mathcal{R}_\omega \vert  := T_1 \cdots T_{d_\omega}$. Define
\[
\widehat{\cH}_\hbar := (\Op_\hbar(\cH_1),\hdots,\Op_\hbar(\cH_{d_\omega})),
\]
with:
\[
\Op_\hbar(\cH_n)=v_n\sum_{j=1}^d \nu_{n,j}( - \hbar^2 \partial_{x_j}^2 +  x_j^2 ),\quad n=1,\hdots,d_\omega.
\]
The following result precises remainders and extends to harmonic oscillators with general frequency vectors \cite{DBievre93} and \cite[Proposition 5]{Ojeda_Villegas10}.
\medskip

\begin{lemma}
\label{concentration_minimal_set}
For every $E\in\Sigma_\cH$ there exist $\hbar_0 > 0$ and a sequence $(\Lambda_\hbar)_{0<\hbar\leq \hbar_0}$ in $\IR^{d_\omega}$ with $\Lambda_\hbar\to E$ as $\hbar \to 0^+$ satisfying the following. For every $z_0 \in \cH^{-1}(E)$ there exists $(\Psi_{\hbar} )_{0<\hbar\leq \hbar_0}$ such that
\[
\widehat{\cH}_\hbar\, \Psi_\hbar = \Lambda_\hbar \, \Psi_\hbar,\quad \Vert \Psi_\hbar \Vert_{L^2(\R^d)} = 1,
\]
and, for every $a\in\cC_c^\infty(\R^{2d})$,
\begin{equation}
\label{e:best_concentration}
\big \langle \Psi_{\hbar}, \Op_{\hbar}(a) \Psi_{\hbar} \big \rangle_{L^2(\R^d)} = \frac{1}{\vert \mathcal{R}_\omega \vert }\int_{\mathcal{R}_\omega}  a\big( \phi^{\mathcal{H}_1}_{t_1} \circ \cdots \circ \phi_{t_{d_\omega}}^{\mathcal{H}_{d_\omega}}(z_0) \big)dt_1 \cdots dt_{d_\omega}  + O(\hbar^{1/2}),\quad \hbar\to 0^+.
\end{equation}
\end{lemma}
\begin{proof}
Let $\nu(n) := v_n (\nu_{n,1}+ \ldots +  \nu_{n,d})$, we have the following explicit expression for the spectrum of each $\Op_\hbar(\mathcal{H}_n)$:
\begin{equation}
\label{e:spectrum}
\operatorname{Sp}\big( \Op_\hbar(\mathcal{H}_n )\big) = \left \{ \lambda^n_{k,\hbar} = v_n\cL_{\nu_n}(\hbar k)+\frac{\hbar \nu(n)}{2}\,:\, k \in\Z^d_+ \right \},\quad n=1,\hdots,d_\omega.
\end{equation}
The fact that each flow $\phi^{\mathcal{H}_n}_t$ is periodic implies that the spacing between consecutive (and large enough) eigenvalues of $\Op_\hbar(\mathcal{H}_n)$ is constant. If $E_n\neq 0$, let $\sigma_n := \operatorname{sign} E_n$. We claim that there exists an integer $N_n>0$ such that
\begin{equation}
\label{e:eigenvalue}
\lambda^n_{\hbar,N}:=\hbar\left(\frac{2\pi \sigma_n}{T_n}  N+\frac{\nu(n)}{2}\right),\quad N\in \mathbb{N},\,N\geq N_n,
\end{equation}
are precisely the eigenvalues of $\Op_\hbar(\mathcal{H}_n)$ that are greater or equal (if $\sigma_n \geq 0$), or less or equal (if $\sigma_n < 0$) to $\lambda^n_{\hbar,N_n}$. 

To see this, recall  (see Remark \ref{r:period}) that there exists a least positive integer $K_n$ such that 
$k(n):= K_n [\nu_n]^{-1}\nu_{n}  \in \IZ^d$ 
(the period of $\phi^{\mathcal{H}_n}_t$ is then $T_n =\frac{ 2\pi K_n}{|v_n| [\nu_{n}]} $). The components of $k(n)$ are integers that are relatively prime. Therefore, there exists $N_n\in\N$ such that, for every integer $N\geq N_n$, there exists $k\in\IZ^d_+$ such that $k\cdot k(n)= \sigma_n N$ (the smallest $N_n$ with this property is called the Frobenius number of the family $\{ k_j(n)\}_{j\in\{1,\hdots,d\}}$). The claim then follows from  \eqref{e:spectrum}.

With this in mind, we set for $n =1, \ldots, d_\omega$:
\begin{align*}
N(\hbar,n) & :=\left\lceil \frac{T_n}{2\pi}\left(\frac{E_n}{\hbar} -\frac{\nu(n)}{2}\right)\right\rceil,\text{ if }E_n\neq 0, \\[0.2cm]
N(\hbar,n) & :=0,\text{ if }E_n= 0,
\end{align*}
and, choosing $\hbar_0>0$ in order to have  $N(\hbar,n)\geq N_n$, for every $n=1,\hdots,d_\omega$, we write, for $0<\hbar<\hbar_0$: 
\begin{equation}\label{e:defeigv}
\lambda^n_\hbar:=\lambda^n_{\hbar,N(\hbar,n)}=\hbar\left(\frac{2\pi \sigma_n}{T_n}N(\hbar,n)+\frac{\nu(n)}{2}\right),\quad \text{ which satisfies } |E_n-\lambda^n_\hbar|< \frac{2\pi}{T_n}\hbar.
\end{equation}
We now write
\[
\Psi_0^\hbar(x):=\frac{1}{(\pi\hbar)^{d/4}}e^{-\frac{|x|^2}{2\hbar}},
\]
and, provided $z_0=(x_0,\xi_0)$, define the coherent state
\begin{equation}\label{e:coherent_state}
\Psi^\hbar_{z_0}(x):= e^{-\frac{i \xi_0\cdot x_0 }{2\hbar}}e^{\frac{i \xi_0\cdot x}{\hbar}}\Psi_0^\hbar(x-x_0).
\end{equation}
Generalizing the idea of \cite{DBievre93}, consider the time average of $\Psi_{z_0}^\hbar$: let $\Lambda_\hbar := (\lambda_\hbar^1, \ldots, \lambda_\hbar^{d_\omega})$, and  define
\begin{equation}
\label{e:non_normalized_candidate}
\langle\Psi^\hbar_{z_0}\rangle := \frac{1}{\vert \mathcal{R}_\omega \vert } \int_{\mathcal{R}_\omega} e^{i \frac{ \tau \cdot \Lambda_\hbar}{\hbar}} e^{-i \frac{\tau \cdot \widehat{\mathcal{H}}_\hbar }{\hbar}} \Psi^\hbar_{z_0}\,d\tau.
\end{equation}
Since $e^{-i \frac{T_n}{\hbar}\Op_\hbar(\mathcal{H}_n)}=e^{-i\frac{T_n \nu(n)}{2}}\operatorname{Id}$, it follows that 
\[
\widehat{\cH}_\hbar \langle\Psi^\hbar_{z_0}\rangle = \Lambda_\hbar \langle\Psi^\hbar_{z_0}\rangle.
\]
On the other hand, it is known (see for instance \cite[Sect. 23.4.2]{Woit_17} or \cite[Chpt. 3]{Robert12}) that
\begin{equation}
e^{-i \frac{ t\Op_\hbar(\mathcal{H}_n)}{\hbar}} \Psi^\hbar_{z_0} = e^{-i \frac{t \nu(n)}{2}} \Psi^\hbar_{\phi_t^{\mathcal{H}_n}(z_0)};
\end{equation}
which allows us to write the time-averaged coherent state solely in terms of the classical dynamics. Defining 
\begin{equation}\label{e:multf}
\Phi_{z_0}(\tau) := \phi_{t_1}^{\mathcal{H}_1} \circ \cdots \circ \phi_{t_{d_\omega}}^{\mathcal{H}_{d_\omega}}, 
\end{equation}
and denoting
\begin{align*}
\Theta_\hbar := 2\pi \left( \frac{N(\hbar,1)}{T_1}, \ldots, \frac{N(\hbar, d_\omega)}{T_{d_\omega}} \right),
\end{align*} 
we have
\[
\langle\Psi^\hbar_{z_0}\rangle := \frac{1}{\vert \mathcal{R}_\omega \vert } \int_{\mathcal{R}_\omega} e^{i \tau \cdot \Theta_\hbar} \Psi^\hbar_{\Phi_{z_0}(\tau)}\,d\tau.
\]
As we will see now, the eigenstates $\langle\Psi^\hbar_{z_0}\rangle$ have yet to be normalized (except of course when $z_0=0$, for which $\langle\Psi^\hbar_{z_0}\rangle=\Psi^\hbar_{z_0}$). The rest of the proof is devoted to analysing, for any $a\in\cC^\infty(\R^{2d})$ such that all its derivatives are bounded, the behavior as $\hbar\to 0^+$ of the oscillatory integral:
\begin{equation}\label{e:mesav}
\big \langle \langle\Psi^\hbar_{z_0}\rangle, \Op_{\hbar}(a) \langle\Psi^\hbar_{z_0}\rangle \big \rangle_{L^2(\R^d)} 
=\int_{\mathcal{R}_\omega \times \mathcal{R}_\omega }e^{i (\tau - \tau') \cdot \Theta_\hbar} 
\big \langle  \Psi^\hbar_{\Phi_{z_0}(\tau')}, \Op_{\hbar}(a)  \Psi^\hbar_{\Phi_{z_0}(\tau)}\big \rangle_{L^2(\R^d)}
\frac{d\tau \, d \tau'}{\vert \mathcal{R}_\omega \vert^2}.
\end{equation}
We next give a stationary phase-type argument that is well suited to the particular structure of this integral.

We first assume that $\rk d \mathcal{H}_{z_0} = d_\omega$; therefore $\Phi_{z_0}|_{\cR_\omega}$ is injective. Start by noting that the inner product in \eqref{e:mesav}, is the cross-Wigner distribution of two phase-space translates of the same Gaussian $\Psi^\hbar_0$ acting on the test function $a$. Define:
\[
\varphi(\tau,\tau',z):=-\frac{\sigma(\Phi_{z_0}(\tau),\Phi_{z_0}(\tau'))}{2}+\sigma(\Phi_{z_0}(\tau)-\Phi_{z_0}(\tau'),z),
\] 
where $\sigma(\cdot,\cdot)$ is the canonical symplectic form in $\R^{2d}$. 
A direct computation then shows:
\begin{equation}
\label{e:cross_with_translation}
\big \langle  \Psi^\hbar_{\Phi_{z_0}(\tau')}, \Op_{\hbar}(a)  \Psi^\hbar_{\Phi_{z_0}(\tau)}\big \rangle_{L^2(\R^d)}=\int_{\R^{2d}}a(z) e^{i\frac{\varphi(\tau,\tau',z)}{\hbar}} W^\hbar_{\Psi^\hbar_0}\left(z-\frac{\Phi_{z_0}(\tau)+\Phi_{z_0}(\tau')}{2}\right)\,dz,
\end{equation}
where $W^\hbar_{\Psi^\hbar_0}$, the Wigner distribution of $\Psi^\hbar_0$, is:
\[
W^\hbar_{\Psi^\hbar_0}(z)=\frac{1}{\hbar^d}G\left(\frac{z}{\sqrt{\hbar}}\right),\quad G(z):=\frac{1}{\pi^d}e^{-|z|^2}.
\]
Inserting this into \eqref{e:mesav} gives:
\begin{equation}\label{e:mesav2}
\int_{\mathcal{R}_\omega \times \mathcal{R}_\omega}\int_{\R^{2d}}e^{i (\tau - \tau') \cdot \Theta_\hbar} e^{i\frac{\varphi(\tau', \tau,-\sqrt{\hbar}z)}{\hbar}} a\left(\frac{\Phi_{z_0}(\tau)+\Phi_{z_0}(\tau')}{2}+\sqrt{\hbar}z\right)G(z)\,dz\frac{d\tau' \, d\tau}{\vert \mathcal{R}_\omega \vert^2}.
\end{equation}
In order to simplify this expression, note first that, since the multi-flow $\Phi_{z_0}(\tau)$ is linear in $z_0$ and Hamiltonian:
\begin{equation}\label{e:symp1}
\frac{\sigma(\Phi_{z_0}(\tau),\Phi_{z_0}(\tau'))}{2}=\frac{\sigma(\Phi_{z_0}(\tau - \tau'),z_0)}{2}=-(\tau - \tau') \cdot \mathcal{H}(z_0)+ g(\tau - \tau'),
\end{equation}
where $g$ is analytic and $g(\tau) = O(\vert \tau \vert^3)$ as $\vert \tau \vert \to 0$. 
Second, by \eqref{e:defeigv}, we have
\begin{equation}\label{e:eigenvapp}
m_\hbar^n:=\frac{2\pi  \sigma_n N(\hbar,n)}{T_n}-\frac{E_n}{\hbar}=\frac{\lambda_\hbar^n-E_n}{\hbar}-\frac{\nu(n)}{2},\quad |m_\hbar^n|\leq \frac{2\pi}{T_n}+\frac{\nu(n)}{2}.
\end{equation} 
Denoting $m_\hbar = (m_\hbar^1, \ldots, m_\hbar^{d_\omega})$, performing a change of variables, and using \eqref{e:symp1}, \eqref{e:eigenvapp} we conclude that  \eqref{e:mesav2} is equal to:
\begin{equation}\label{e:mesav3}
\int_{\mathcal{R}_\omega \times \mathcal{R}_\omega}\int_{\R^{2d}}e^{i \tau \cdot m_\hbar} e^{i\frac{g(\tau)}{\hbar}} e^{\frac{i}{\hbar}\sigma\left(\Phi_{z_0}(\tau + \tau') -\Phi_{z_0}(\tau') ,z \right)}  a\left(\frac{\Phi_{z_0}(\tau+\tau')+\Phi_{z_0}(\tau')}{2}+\sqrt{\hbar}z\right)G(z)\,dz\frac{d\tau' \, d\tau}{|\cR_\omega|^2}.
\end{equation}
Now we use that:
\[
\sigma\left(\Phi_{z_0}(\tau + \tau') -\Phi_{z_0}(\tau') ,z \right)=- \tau \cdot d \mathcal{H}_{\Phi_{z_0}(\tau')}(z) + r(\tau,\tau') \cdot z,
\]
where $r$ is analytic in both arguments and $r(\tau,\tau') = O(\vert \tau \vert^2)$, and that:
\[
e^{ir(\tau,\tau')\cdot \sqrt{\hbar}z}a\left(\frac{\Phi_{z_0}(\tau + \tau')+\Phi_{z_0}(\tau')}{2}+\sqrt{\hbar}z\right)=a\left(\frac{\Phi_{z_0}(\tau + \tau')+\Phi_{z_0}(\tau')}{2}\right)+O(\sqrt{\hbar})R_\hbar,
\]
with a remainder such that $R_\hbar G$ is uniformly bounded in $\cS(\R^{2d})$ with respect to $\hbar,\tau,\tau'$.
Then, denoting
$$
\mathcal{R}_\omega^\hbar := \left[ 0, \frac{T_1}{\hbar} \right) \times \cdots \times \left[0, \frac{T_{d_\omega}}{\hbar} \right),
$$ 
\eqref{e:mesav3} can be transformed in:
\begin{align*}
\hbar^{\frac{d_\omega}{2}} \int_{\mathcal{R}_\omega^\hbar} \int_{\mathcal{R}_\omega} \int_{\R^{2d}}e^{i \sqrt{\hbar}( \tau \cdot m_\hbar+g(\tau))}  &e^{-i\tau d\mathcal{H}_{\Phi_{z_0}(\tau')}(z)}  \\ & a\left(\frac{\Phi_{z_0}(\tau' + \sqrt{\hbar} \tau) + \Phi_{z_0}(\tau')}{2}\right)G(z)\,dz\frac{d\tau' \, d\tau}{ \vert \mathcal{R}_\omega \vert^2}+O(\hbar)=
\end{align*}
\begin{equation}\label{e:mesav4}
\frac{1}{\vert \mathcal{R}_\omega \vert}\int_{\mathcal{R}_\omega} \hbar^{\frac{d_\omega}{2}} \int_{\mathcal{R}_\omega^\hbar} e^{i \sqrt{\hbar}(\tau \cdot m_\hbar+g(\tau))}  a\left(\frac{\Phi_{z_0}(\tau' + \sqrt{\hbar} \tau) + \Phi_{z_0}(\tau')}{2}\right)\widehat{G}([ d\mathcal{H}_{\Phi_{z_0}(\tau')}]^T \tau)\frac{d\tau}{\vert \mathcal{R}_\omega \vert}d\tau' +O(\hbar).
\end{equation}
At this point, observe that $\widehat{G}=\pi^dG^{1/4}$ and that 
$$
(d\mathcal{H}_{\Phi_{z_0}(\tau')}) (d\mathcal{H}_{\Phi_{z_0}(\tau')})^T= (d\mathcal{H}_{z_0})(d\mathcal{H}_{z_0})^T =: \mathcal{G}_{z_0},
$$ 
for every $\tau' \in \mathcal{R}_\omega$. This implies:
\begin{align*}
 \int_{\mathcal{R}_\omega^\hbar} e^{i \sqrt{\hbar}(\tau \cdot m_\hbar+g(\tau))}  a\left(\frac{\Phi_{z_0}(\tau' + \sqrt{\hbar} \tau) + \Phi_{z_0}(\tau')}{2}\right)& \widehat{G}\big( [d\mathcal{H}_{\Phi_{z_0}(\tau')}]^T \tau \big)\frac{d\tau}{\vert \mathcal{R}_\omega \vert}=\\&a(\Phi_{z_0}(\tau'))\int_{\R_+^{d_\omega}}  e^{-\frac{ \vert (d\mathcal{H}_{z_0})^T \tau \vert^2}{4}}\frac{d\tau }{\vert \mathcal{R}_\omega \vert } +O(\sqrt{\hbar}).
\end{align*}
Inserting this identity back in \eqref{e:mesav4}, we conclude that, taking $a=1$:
\[
c_\hbar:=\frac{\vert \mathcal{R}_\omega\vert \sqrt{\det \mathcal{G}_{z_0}}}{(\pi\hbar)^{\frac{d_\omega}{2}}}\|\langle \Psi^\hbar_{z_0}\rangle\|_{L^2(\R^d)}^2\To 1,\quad \hbar\to 0^+.
\]
and that 
\[
\Psi_\hbar :=\left(\frac{\vert \mathcal{R}_\omega\vert \sqrt{\det \mathcal{G}_{z_0}}}{c_\hbar (\pi\hbar)^{\frac{d_\omega}{2}}}\right)^{1/2}\langle \Psi^\hbar_{z_0}\rangle
\]
are the desired normalized eigenstates satisfying \eqref{e:best_concentration}. 

We now turn to the case $1 \leq \rk d \mathcal{H}_{z_0} = d_0 < d_\omega$ . The argument is very similar, but the normalization constant turns out to be different. We define, for any $u \in \R^d$,
$$
\pi_{z_0}(u) = \big( \mathbf{1}_{(0,\infty)}(H_1(z_0))u_1, \ldots, \mathbf{1}_{(0,\infty)}(H_d(z_0))u_d \big),
$$
Since $\pi_{z_0}(\omega)=\sum_{j=1}^{d_\omega}v_n\pi_{z_0}(\nu_n)$ we have $d_0 = \dim S_{\pi_{z_0}(\omega)}$. Modulo permutation of the indexes, for every $ d_0 +1 \leq n \leq d_\omega$, there exists $b_n \in \mathbb{Q}^{d_0}$ such that
$$
\pi_{z_0}(\nu_n) = \sum_{l=1}^{d_0} b_{n,l} \pi_{z_0}(\nu_l).
$$
Then $\pi_{z_0}(\omega) = \widetilde{v}_1 \pi_{z_0}(\nu_1) + \cdots + \widetilde{v}_{d_0} \pi_{z_0}(\nu_{d_0})$, where
$$
\widetilde{v}_l := v_l + \sum_{n=d_0+1}^{d_\omega} b_{n ,l} v_n,\quad l= 1, \ldots, d_0.
$$
We next define, for $1 \leq l \leq d_0$, $\widetilde{\mathcal{H}}_l := \widetilde{v}_l \pi_{z_0}(\nu_l) \cdot (H_1, \ldots, H_d)$ and write $\widetilde{\mathcal{H}} = (\widetilde{\mathcal{H}}_1, \ldots, \widetilde{\mathcal{H}}_{d_0})$. Notice that, for $1 \leq n \leq d_0$,
\begin{align}
\label{e:implicit_1}
\mathcal{H}_n(z_0) & = \frac{v_n}{\widetilde{v}_n} \widetilde{\mathcal{H}}_n(z_0),  
\end{align}
while, for $d_0 + 1 \leq n \leq d_\omega$,
\begin{align}
\label{e:implicit_2}
\mathcal{H}_n(z_0) & 
  = v_n \sum_{l=1}^{d_0} \frac{b_{n,l}}{\widetilde{v}_l} \widetilde{\mathcal{H}}_l(z_0).
\end{align} 
Then, under our assumptions, 
$$
\rk d \widetilde{\mathcal{H}}_{z_0} = \rk d \mathcal{H}_{z_0} = d_0.
$$ 
We now define the rectangle $\widetilde{R}_\omega := [0, \widetilde{T}_1) \times \cdots \times [0,\widetilde{T}_{d_0})$, where
$$
\widetilde{T}_l := \frac{2\pi \widetilde{K}_l}{\vert \widetilde{v}_l \vert [\pi_{z_0}(\nu_l)]},
$$
and $\widetilde{K}_l$ is the least positive integer such that $\widetilde{k}(l) = \widetilde{K}_l [ \pi_{z_0}(\nu_l)]^{-1} \pi_{z_0}( \nu_l )\in \mathbb{Z}^d$. Let us  define also $\widetilde{\nu}(l) := \widetilde{v}_l( \mathbf{1}_{(0,\infty)}(H_1(z_0)) \nu_{l,1} + \cdots + \mathbf{1}_{(0,\infty)}(H_d(z_0)) \nu_{l,d})$, and, setting $\widetilde{E}_l = \widetilde{\mathcal{H}}_l(z_0)$, let
\begin{align*}
\widetilde{N}(\hbar,l) & := \left \lceil \frac{\widetilde{T}_l}{2 \pi} \left( \frac{\widetilde{E}_l}{\hbar} - \frac{\widetilde{\nu}(l)}{2}\right) \right \rceil, \quad \text{if } \widetilde{E}_l \neq 0, \\[0.2cm]
\widetilde{N}(\hbar,l) & := 0, \quad \text{if } \widetilde{E}_l = 0.
\end{align*}
For $1 \leq l \leq d_0$, we define
$$
\widetilde{\lambda}_{\hbar}^l := \hbar\left(\frac{2\pi \sigma_l}{\widetilde{T}_l}\widetilde{N}(\hbar,l)+\frac{\widetilde{\nu}(l)}{2} \right).
$$
By construction, in correspondence with \eqref{e:implicit_1} and \eqref{e:implicit_2}, we have
\begin{align*}
\lambda_{\hbar}^n & = \frac{v_n}{\widetilde{v}_n} \widetilde{\lambda}_\hbar^n, \quad 1 \leq n \leq d_0, \\[0.2cm]
\lambda_{\hbar}^n & = v_n \sum_{l=1}^{d_0} \frac{b_{n,l}}{\widetilde{v}_l} \widetilde{\lambda}_\hbar^l, \quad d_0 + 1 \leq n \leq d_\omega.
\end{align*}
Taking $\widetilde{\Lambda}_\hbar = (\widetilde{\lambda}_\hbar^1, \ldots , \widetilde{\lambda}_\hbar^{d_0})$, this implies that the averaged coherent-state given by
$$
\langle\Psi^\hbar_{z_0}\rangle := \frac{1}{\vert \widetilde{\mathcal{R}}_\omega \vert } \int_{\widetilde{\mathcal{R}}_\omega} e^{i \frac{ \tau \cdot \widetilde{\Lambda}_\hbar}{\hbar}} e^{-i \frac{\tau \cdot \Op_\hbar(\widetilde{\mathcal{H}}) }{\hbar}} \Psi^\hbar_{z_0}\,d\tau
$$
coincides with \eqref{e:non_normalized_candidate} up to multiplication by a constant. In particular,
$$
\widehat{\mathcal{H}}_\hbar \langle\Psi^\hbar_{z_0}\rangle = \Lambda_\hbar \langle\Psi^\hbar_{z_0}\rangle.
$$
The rest of the proof mimics the one before, with normalized states replaced by
$$
\Psi_\hbar :=\left(\frac{\vert \widetilde{\mathcal{R}}_\omega\vert \sqrt{\det \widetilde{\mathcal{G}}_{z_0}}}{c_\hbar (\pi\hbar)^{\frac{d_0}{2}}}\right)^{1/2}\langle \Psi^\hbar_{z_0}\rangle,
$$
where $\widetilde{\mathcal{G}}_{z_0} = (d\widetilde{\mathcal{H}}_{z_0})(d\widetilde{\mathcal{H}}_{z_0})^T$, and now the normalizing constant $c_\hbar$ is replaced by
$$
c_\hbar:=\frac{\vert \widetilde{\mathcal{R}}_\omega\vert \sqrt{\det \widetilde{\mathcal{G}}_{z_0}}}{(\pi\hbar)^{\frac{d_0}{2}}}\|\langle \Psi^\hbar_{z_0}\rangle\|_{L^2(\R^d)}^2,
$$
Finally, for any $a \in \mathcal{C}_c^\infty(\R^{2d})$, we obtain
$$
\big \langle \Psi_{\hbar}, \Op_{\hbar}(a) \Psi_{\hbar} \big \rangle_{L^2(\R^d)} = \frac{1}{\vert \widetilde{\mathcal{R}}_\omega \vert }\int_{\widetilde{\mathcal{R}}_\omega}  a\big( \phi^{\tilde{\mathcal{H}}_1}_{t_1} \circ \cdots \circ \phi_{t_{d_0}}^{\tilde{\mathcal{H}}_{d_0}}(z_0) \big)dt_1 \cdots dt_{d_0}  + O(\hbar^{1/2}),\quad \hbar\to 0^+.
$$
By construction, we also have that 
$$
\frac{1}{\vert \widetilde{\mathcal{R}}_\omega \vert }\int_{\widetilde{\mathcal{R}}_\omega}  a\big( \phi^{\tilde{\mathcal{H}}_1}_{t_1} \circ \cdots \circ \phi_{t_{d_0}}^{\tilde{\mathcal{H}}_{d_0}}(z_0) \big)dt_1 \cdots dt_{d_0} = \frac{1}{\vert \mathcal{R}_\omega \vert }\int_{\mathcal{R}_\omega}  a \circ \Phi_{z_0}(\tau) d\tau.
$$
This completes the proof of the lemma.

\end{proof}

\section{Proof of Theorem \ref{non-perturbed-eigenfunctions}} 
\label{proof}


Let us show that for any choice of $E\in\Sigma_\cH$ one has that 
\[
\cM_{E}(\cH):=\{\mu\in\cM(H)\,:\,\supp \mu\subseteq \cH^{-1}(E)\}\subseteq \cM(\widehat{H}_\hbar).
\]
To this aim, note that since $\cM_{E}(\cH)$ is convex and compact for the weak-$\star$ topology, the Krein-Milman theorem ensures that $\cM_{E}(\cH)$ is the closure of the convex hull of the set of orbit measures $\delta_{\mathcal{T}_\omega(z_0)}$, where
\begin{equation}\label{e:atom}
\mathcal{T}_\omega(z_0) := \{ \Phi_{z_0}(\tau) \in \mathcal{H}^{-1}(E) \, : \, \tau \in \mathcal{R}_\omega \}, \quad z_0 \in \mathcal{H}^{-1}(E).
\end{equation}
Therefore, it suffices to show that any measure $\mu$ that is a convex combination of distinct orbit measures $\delta_{\cT_1},\hdots,\delta_{\cT_r}$, with $\cT_j:=\cT_\omega(z_j)$, $z_j\in\cH^{-1}(E)$, belongs to $\cM(\widehat{H}_\hbar)$. 
Apply Lemma \ref{concentration_minimal_set} to  ensure the existence of a sequence of eigenvalues $(\lambda_\hbar^1,\hdots,\lambda^{d_\omega}_\hbar) \to E$ as $\hbar\to 0^+$, and normalized eigenfunctions $\psi_\hbar^j \in L^2(\R^d)$, $j=1,\hdots,r$ such that, 
\[
\widehat{H}_\hbar\psi_\hbar^j = \lambda_\hbar \psi_\hbar^j, \quad \lambda_\hbar:=\lambda_\hbar^1+\hdots+\lambda^{d_\omega}_\hbar\to 1,\;\text{ as }\hbar\to 0^+,
\]
and, for every $a\in\cC^{\infty}_c(\R^{2d})$,
\[
\lim_{\hbar\to 0^+}\big \langle \psi_\hbar^j, \Op_{\hbar}(a) \psi_\hbar^j \big \rangle_{L^2(\R^d)}=\int_{\R^{2d}}a(x,\xi)\, \delta_{\mathcal{T}_j} (dx,d\xi).
\]
Suppose that $\mu=\sum_{j=1}^r\alpha_j\delta_{\cT_j}$ with $\alpha_j\in(0,1)$; since the measures $\delta_{\cT_j}$ are mutually singular, it follows (see for instance \cite[Proposition 3.3]{GerardMesuresSemi91}) that 
\[
\lim_{\hbar\to 0^+}\big \langle \psi_\hbar^j, \Op_{\hbar}(a) \psi_\hbar^k \big \rangle_{L^2(\R^d)} =0,\quad j\neq k.
\]
This shows that the sequence of eigenfunctions $\psi_\hbar:=\sum_{j=1}^r\sqrt{\alpha_j}\psi_\hbar^j$ has $\mu$ as a semiclassical measure. Even though the sequence $(\psi_\hbar)$ is not normalized, its norm tends to one so one concludes that $\mu\in\cM(\widehat{H}_\hbar)$.

Conversely, suppose that $\mu\in\cM(\widehat{H}_\hbar)$; let us show that there exists a $E\in\Sigma_\cH$ such that $\mu\in\cM_E(\cH)$. Suppose that $\mu$ is obtained as the weak-$\star$ limit of the sequence of Wigner functions of $(\psi_{\hbar_k})$, where $\hbar_k\to 0^+$, and:
\[
\widehat{H}_{\hbar_k}\psi_{\hbar_k} = \lambda_{\hbar_k} \psi_{\hbar_k},\quad \|\psi_{\hbar_k}\|_{L^2(\IR^d)}=1,\quad \lim_{k\to\infty}\lambda_{\hbar_k}=1. 
\]
Note that the functions $\psi_{\hbar_k}$ are linear combinations of joint eigenvectors of $-\hbar_k\partial_{x_j}+x_j^2$, $j=1,\hdots,d$, and therefore also of joint eigenvectors of $\Op_{\hbar_k}(\cH_n)$, $n=1,\hdots,d_\omega$. It follows from \eqref{e:spectrum} that the eigenvalues of $\Op_\hbar(\cH_n)$ are contained in $\hbar \left\la v_n \right\ra_\IQ $. The decomposition \eqref{e:decper} then implies that the eigenvalues of $\widehat{H}_\hbar$ are contained in $\hbar S_\omega$. Since $\{ v_n\}_{n=1, \ldots, d_\omega}$ form a basis of $S_\omega$, it follows that there exist unique $\lambda_{\hbar_k}^n\in\Spec{(\Op_{\hbar_k}(\cH_n))}$, $n=1,\hdots,d_\omega$, such that:
\begin{equation}
\label{e:eigenvalue_decomposition}
\lambda_{\hbar_k}=\lambda_{\hbar_k}^1+\hdots+\lambda_{\hbar_k}^{d_\omega},
\end{equation}
and necessarily  the eigenfunctions $\psi_\hbar$ must satisfy:
\begin{equation}\label{e:joineing}
\Op_{\hbar_k}(\cH_n)\psi_{\hbar_k}^n=\lambda_{\hbar_k}^n\psi_{\hbar_k}.
\end{equation}
Note that $\lambda_{\hbar_k}=\hbar_k\left(\omega\cdot l_k+\frac{|\omega|_1}{2}\right)$ for some sequence $(l_k)$ in $\IZ^d_+$ such that $(\hbar_kl_k)$ is bounded (recall that the frequencies $\omega_j$ are positive). This implies that each sequence $(\lambda_{\hbar_k}^n)$ is bounded. After extracting a subsequence, that we do not relabel, we can assume that 
\[
\forall n=1,\hdots,d_\omega,\;\lim_{k\to\infty}\lambda_{\hbar_k}^n= E_n,\quad\text{ and }\quad \sum_{n=1}^{d_\omega}E_n= 1. 
\]
Identity \eqref{e:joineing} then ensures that:
\begin{equation}\label{e:mloc}
(\cH_n-E_n)\mu=0,\quad X_{\cH_n}\mu=0, \quad n=1,\hdots, d_\omega,
\end{equation}
where $ X_{\cH_n}$ denotes the Hamiltonian vector field of $ \cH_n$. Note that that $E=(E_1,\hdots,E_{d_\omega})\in \Sigma_\cH$ as a consequence of this, otherwise \eqref{e:mloc} would imply $\mu=0$. Therefore, $\mu\in\cM_E(\cH)$ as claimed.
This concludes the proof of the theorem.


\end{document}